\documentclass[11pt,oneside]{amsart}

\usepackage{fullpage}

\usepackage{amsthm,amsfonts,amssymb,amsmath,amsxtra}
\usepackage[all]{xy}
\SelectTips{cm}{}
\usepackage{tikz}
\usepackage{verbatim}
\usepackage{mathrsfs}

\RequirePackage{xspace}
\RequirePackage{etoolbox}
\RequirePackage{varwidth}
\RequirePackage{enumitem}
\RequirePackage{tensor}
\RequirePackage{mathtools}
\RequirePackage{longtable}
\RequirePackage{multirow}
\RequirePackage{makecell}
\RequirePackage{bigints}

\usepackage[backref=page]{hyperref} %

\newcommand{\Ch}{{\mathrm{Ch}}}
\DeclareMathOperator{\charac}{char}

\DeclareMathOperator{\Gal}{Gal}

\newcommand{\GL}{\mathrm{GL}}

\DeclareMathOperator{\Hom}{Hom}

\newcommand{\Ind}{{\mathrm{Ind}}}

\DeclareMathOperator{\ord}{ord}

\DeclareMathOperator{\rank}{rank}

\renewcommand{\Re}{{\mathrm{Re}}}

\DeclareMathOperator{\Res}{Res}

\DeclareMathOperator{\Sym}{Sym}

\newcommand{\ov}{\overline}

\newcommand{\lra}{\longrightarrow}

\newenvironment{altenumerate}
   {\begin{list}
      {(\theenumi) }
      {\usecounter{enumi}
       \setlength{\labelwidth}{0pt}
       \setlength{\labelsep}{0pt}
       \setlength{\leftmargin}{0pt}
       \setlength{\itemsep}{\the\smallskipamount}
       \renewcommand{\theenumi}{\roman{enumi}}
      }}
   {\end{list}}
\setitemize[0]{leftmargin=*,itemsep=\the\smallskipamount}
\setenumerate[0]{leftmargin=*,itemsep=\the\smallskipamount}

\renewcommand{\to}{%
   \ifbool{@display}{\longrightarrow}{\rightarrow}%
   }
\let\shortmapsto\mapsto
\renewcommand{\mapsto}{%
   \ifbool{@display}{\longmapsto}{\shortmapsto}%
   }
\newlength{\olen}
\newlength{\ulen}
\newlength{\xlen}
\newcommand{\xra}[2][]{%
   \ifbool{@display}%
      {\settowidth{\olen}{$\overset{#2}{\longrightarrow}$}%
       \settowidth{\ulen}{$\underset{#1}{\longrightarrow}$}%
       \settowidth{\xlen}{$\xrightarrow[#1]{#2}$}%
       \ifdimgreater{\olen}{\xlen}%
          {\underset{#1}{\overset{#2}{\longrightarrow}}}%
          {\ifdimgreater{\ulen}{\xlen}%
             {\underset{#1}{\overset{#2}{\longrightarrow}}}
             {\xrightarrow[#1]{#2}}}}%
      {\xrightarrow[#1]{#2}}
   }
\makeatother
\newcommand{\xyra}[2][]{%
   \settowidth{\xlen}{$\xrightarrow[#1]{#2}$}%
   \ifbool{@display}%
      {\settowidth{\olen}{$\overset{#2}{\longrightarrow}$}%
       \settowidth{\ulen}{$\underset{#1}{\longrightarrow}$}%
       \ifdimgreater{\olen}{\xlen}%
          {\mathrel{\xymatrix@M=.12ex@C=3.2ex{\ar[r]^-{#2}_-{#1} &}}}%
          {\ifdimgreater{\ulen}{\xlen}%
             {\mathrel{\xymatrix@M=.12ex@C=3.2ex{\ar[r]^-{#2}_-{#1} &}}}
             {\mathrel{\xymatrix@M=.12ex@C=\the\xlen{\ar[r]^-{#2}_-{#1} &}}}}}%
      {\mathrel{\xymatrix@M=.12ex@C=\the\xlen{\ar[r]^-{#2}_-{#1} &}}}%
   }
\makeatletter
\newcommand{\xla}[2][]{%
   \ifbool{@display}%
      {\settowidth{\olen}{$\overset{#2}{\longleftarrow}$}%
       \settowidth{\ulen}{$\underset{#1}{\longleftarrow}$}%
       \settowidth{\xlen}{$\xleftarrow[#1]{#2}$}%
       \ifdimgreater{\olen}{\xlen}%
          {\underset{#1}{\overset{#2}{\longleftarrow}}}%
          {\ifdimgreater{\ulen}{\xlen}%
             {\underset{#1}{\overset{#2}{\longleftarrow}}}
             {\xleftarrow[#1]{#2}}}}%
      {\xleftarrow[#1]{#2}}
   }
\newcommand{\isoarrow}{%
   \ifbool{@display}{\overset{\sim}{\longrightarrow}}{\xrightarrow\sim}%
   }
\renewcommand{\lra}{%
   \ifbool{@display}{\longleftrightarrow}{\leftrightarrow}%
   }
\DeclareFontFamily{U}{matha}{\hyphenchar\font45}
\DeclareFontShape{U}{matha}{m}{n}{
      <5> <6> <7> <8> <9> <10> gen * matha
      <10.95> matha10 <12> <14.4> <17.28> <20.74> <24.88> matha12
      }{}
\DeclareSymbolFont{matha}{U}{matha}{m}{n}
\DeclareFontFamily{U}{mathx}{\hyphenchar\font45}
\DeclareFontShape{U}{mathx}{m}{n}{
      <5> <6> <7> <8> <9> <10>
      <10.95> <12> <14.4> <17.28> <20.74> <24.88>
      mathx10
      }{}
\DeclareSymbolFont{mathx}{U}{mathx}{m}{n}

\DeclareMathSymbol{\obot}         {2}{matha}{"6B}

\newtheorem{theorem}{Theorem}[section]

\newtheorem{lemma}[theorem]{Lemma}
\newtheorem {conjecture}[theorem]{Conjecture}

\theoremstyle{definition}
\newtheorem{definition}[theorem]{Definition}

\newtheorem{remark}[theorem]{Remark}

\numberwithin{equation}{theorem}

\renewcommand{\@secnumfont}{\bfseries}
\renewcommand\section{\@startsection{section}{2}%
  \z@{.5\linespacing\@plus.7\linespacing}{-.5em}%
  {\normalfont\bfseries}}

\newcommand{\Chh}{\Ch_{\hom}}
\DeclareMathOperator{\HH}{H}
\newcommand{\Qlb}{\overline{\mathbb{Q}}_\ell}

\makeatletter
\@namedef{subjclassname@2020}{%
  \textup{2020} Mathematics Subject Classification}
\makeatother

\title{A note on Tate's conjectures for abelian varieties}

\author[Chao Li]{Chao Li}
\address{Columbia University, Department of Mathematics, 2990 Broadway,	New York, NY 10027, USA}
\email{chaoli@math.columbia.edu} 

\author[Wei Zhang]{Wei Zhang}
\address{Massachusetts Institute of Technology, Department of Mathematics, 77 Massachusetts Avenue, Cambridge, MA 02139, USA}
\email{weizhang@mit.edu}

\date{\today}


\begin{document}

\maketitle{}

\begin{abstract}
  In this mostly expository note, we explain a proof of Tate's two conjectures \cite{Tat65} for algebraic cycles of arbitrary codimension on certain products of elliptic curves and abelian surfaces over number fields.
\end{abstract}

\section{Statement}

Let $X$ be a smooth projective variety over a finitely generated field $F$. Let $\Ch^r(X)$ be the Chow group of codimension $r$ algebraic cycles of $X$ defined over $F$ modulo rational equivalence. Let $\ov F$ be a separable algebraic closure of $F$ and $\Gamma_F:=\Gal(\ov F/F)$. Tate  {\cite[Conjecture 1]{Tat65}} made the following far-reaching conjecture  (often known as \emph{the Tate conjecture}), relating algebraic cycles and $\Gamma_F$-invariants of the $\ell$-adic cohomology of $X$.

\begin{conjecture}[Tate I]\label{conj:Tate} For any $1\le r\le \dim X$ and for any prime $\ell\ne \charac(F)$, the $\ell$-adic cycle class map $$\Ch^r(X) \otimes \mathbb{Q}_\ell\rightarrow \HH^{2r}(X_{\ov F}, \mathbb{Q}_\ell(r))^{\Gamma_F}$$ is surjective.
\end{conjecture}

 Let $\Chh^r(X)$ be the quotient group of $\Ch^r(X)$ modulo $\ell$-adic homological equivalence. It is further conjectured (and known when $\charac(F)=0$) that $\Chh^r(X)$ is independent of $\ell$, and the $\ell$-adic cycle class map is injective on $\Chh^r(X) \otimes \mathbb{Q}_\ell$ (see \cite[p.97]{Tat65}). In particular, when $\charac(F)=0$, Tate I implies an isomorphism $\Chh^r(X) \otimes \mathbb{Q}_\ell \simeq  \HH^{2r}(X_{\ov F}, \mathbb{Q}_\ell(r))^{\Gamma_F}$ and thus 
 \begin{equation}\label{eq:Chh}
   \rank \Chh^r(X)=\dim\HH^{2r}(X_{\ov F}, \mathbb{Q}_\ell(r))^{\Gamma_F}
 \end{equation}
for any prime $\ell$.

Tate {\cite[Conjecture 2]{Tat65}} further made a conjecture relating algebraic cycles to poles of zeta functions (often known as \emph{the strong Tate conjecture}). When $F$ is a number field, we denote by $L(\HH^{2r}(X)(r),s)$ the (incomplete) $L$-function associated to the compatible system $\{\HH^{2r}(X_{\ov F},\Qlb(r))\}$ of $\Gamma_F$-representations, which converges absolutely for $\Re(s)>1$. Then \cite[Conjecture 2]{Tat65} specializes to the following.

\begin{conjecture}[Tate II] Assume that $F$ is a number field.  Then  for any $1\le r\le \dim X$,  $$\rank\Chh^r(X)=-\ord_{s=1}L(\HH^{2r}(X)(r),s).$$
\end{conjecture}

Tate I for divisors ($r=1$) is known for various $X$, including abelian varieties over any finitely generated fields (\cite{Fal83,  Zar75, Tat66}). Much less is known when $r>1$. We refer to the surveys \cite{Tot17,Mil07, Tat94,Ram89} for a nice summary of known results. The goal of this short note is to present some examples of abelian varieties $X$ over number fields for which Tate's conjectures hold for algebraic cycles in {\em arbitrary} codimension $r$. 

\begin{theorem}[Tate I]\label{thm:main1}Assume that $F$ is finitely generated with $\charac(F)=0$. Then Tate I holds for any abelian variety $X$ over $F$ with simple factors all having dimension $\le 2$.
\end{theorem}

\begin{theorem}[Tate II]\label{thm:main2}
  Assume that $F$ is a number field.  Let $E_1,E_2,E_3,E_4$ be elliptic curves over $F$. Let $A$  be an abelian surface over $F$. Then Tate II  holds for the following cases:
    \begin{altenumerate}
    \item $F$ is totally real or imaginary CM and $X=E_1^{n_1}\times E_2^{n_2}$ for any $n_1\ge1 ,n_2\ge0$,
    \item $F$ is totally real or imaginary CM and $X=E_1^{n_1}\times E_2^{n_2}\times E_3$ for any $n_1\ge1$ and $1\le n_2\le2$.
    \item $F$ is totally real or imaginary CM and $X=E_1^{n_1}\times E_2^{n_2}\times E_3\times E_4$ for any $1\le n_1, n_2\le2$.
    \item $F$ is totally real and $X=A$, $X=A^2$.
    \end{altenumerate}
\end{theorem}

\begin{remark}
It is worth mentioning that the special case when $X=E^n$ is a power of an elliptic curve was considered by Tate himself \cite[p.106]{Tat65}, and played an important role in his formulation of the Sato--Tate conjecture.  
\end{remark}

Theorem \ref{thm:main1} (Tate I) can be deduced from recent theorems on the Hodge conjecture and the Mumford--Tate conjecture  (\cite{RMar08,Lom16}), as mentioned e.g. in \cite[p.284]{Moo17}.  Theorem \ref{thm:main2} (Tate II) can be deduced from more recent potential automorphy theorems (\cite{ACC+18,BCGP21}) and known cases of Langlands functionality, and should also be known to the experts.  All these ingredients are available in more generality, but to illustrate the ideas we do not aim for maximal generality in the statement of the theorems. 

\section{Proof of Theorem \ref{thm:main1} (Tate I)}

Choose an embedding $F\hookrightarrow \mathbb{C}$ and view $F$ as a subfield of $\mathbb{C}$. Since all simple factors of $X$ have dimension $\le2$, the Hodge conjecture for $X_{\mathbb{C}}$ holds (in any codimension $r$) by \cite[Theorem 3.15]{RMar08}. In fact in this case all Hodge classes on $X_\mathbb{C}$ are generated by products of divisor classes.  Also by \cite[Corollary 1.2]{Lom16}, the Mumford--Tate conjecture for $X$ holds. 

Now the desired result follows due to the well-known general fact (see e.g. \cite[\S 6]{CFar16}) that the Mumford--Tate conjecture for the abelian variety $X$ over $F$ together with the Hodge conjecture for $X_\mathbb{C}$ (in codimension $r$) implies Tate I (Conjecture \ref{conj:Tate}) for $X$ (in codimension $r$). In particular all Tate classes on $X$ are also generated by products of divisor classes.

\begin{remark}
 We refer to the references in \cite{RMar08}, \cite{Lom16} for related previous works on the Hodge and Mumford--Tate conjectures.  When $X$ is a product of elliptic curves, the Hodge conjecture was proved in \cite{Mur90} (see also \cite[Appendix~B, \S3]{Lew99}) and the same method should also apply to prove Tate I. 
\end{remark}

\section{Potential automorphy}

Let $F$ be a number field. Let $V=\{V_\ell\}$ and $W=\{W_\ell\}$ be  compatible systems of semisimple $\ell$-adic $\Gamma_F$-representations (e.g., in the sense of strictly compatible systems of $\ell$-adic representations of $\Gamma_F$ defined over $\mathbb{Q}$ of \cite[\S2.8]{BCGP21}). Recall that $V$ is  \emph{potentially automorphic} if there exists a finite Galois extension $L/F$ such that the restriction $V|_{\Gamma_L}$ is automorphic (e.g., in the sense of \cite[Definition 9.1.1]{BCGP21}). We introduce the following variants of potential automorphy.

\begin{definition}\label{def:strong}
  Let $S$ be a nonempty set of rational primes. Let $L/F$ be a finite Galois extension.

  We say that $V$ is \emph{$S$-strongly automorphic over $L$}, if for any subextension $L'/F$ of $L/F$ with $L/L'$ solvable, the following conditions are satisfied:
\begin{altenumerate}
\item  $V|_{\Gamma_{L'}}$ is automorphic.
\item\label{item:irr} Let $\pi$ be an isobaric automorphic representation on $\GL_n(\mathbb{A}_{L'})$ associated to $V|_{\Gamma_{L'}}$ ($n=\dim V$ and $\mathbb{A}_{L'}$ is the ring of ad\`eles of $L'$). Write $\pi=\boxplus_{i=1}^k\pi_i$ as an isobaric direct sum of cuspidal automorphic representations on $\GL_{n_i}(\mathbb{A}_{L'})$ ($n=\sum_{i=1}^k n_i$). Write $V|_{\Gamma_{L'}}=\oplus_{i=1}^k V_{i}$ as the corresponding direct sum decomposition into compatible systems of $\Gamma_{L'}$-representations. Then the  $\ell$-adic $\Gamma_{L'}$-representation $V_{i,\ell}$ ($i=1,\ldots k$) is irreducible for any $\ell\in S$. (Notice that the irreducibility of $V_{i,\ell}$ is conjectured but not known in general).
\end{altenumerate}

We say that $V$ is \emph{$S$-strongly potentially automorphic}, if $V$ is $S$-strongly automorphic over $L$ for some finite Galois extension $L/F$. We say that $V$ is \emph{strongly potentially automorphic}, if $V$ is $S$-strongly potentially automorphic for some Dirichlet density one set $S$.

We say that $V$ and $W$ are \emph{jointly $S$-strongly potentially automorphic}, if $V$ and $W$ are both $S$-strongly automorphic over $L$ for some finite Galois extension $L/F$. We say that $V$ and $W$ are \emph{jointly strongly potentially automorphic}, if $V$ and $W$ are jointly $S$-strongly potentially automorphic for some Dirichlet density one set $S$. 
\end{definition}

\begin{lemma}\label{lem:pole} Let $V=\{V_\ell\}$ and $W=\{W_\ell\}$ be  compatible systems of semisimple $\ell$-adic $\Gamma_F$-representations. Let  $S$ be a nonempty set of rational primes.
  \begin{altenumerate} 
  \item\label{item:1}   Assume that $V$ is $S$-strongly potentially automorphic. Then $L(V,s)$ has meromorphic continuation to all of $\mathbb{C}$, and for any $\ell \in S$, $$\dim V_\ell^{\Gamma_F}=-\ord_{s=1}L(V,s).$$
  \item\label{item:2} Assume that $V$ and $W$ are jointly $S$-strongly potentially automorphic. Then $L(V \otimes W,s)$ has meromorphic continuation to all of $\mathbb{C}$, and for any $\ell \in S$, $$\dim (V_\ell \otimes W_\ell)^{\Gamma_F}=-\ord_{s=1}L(V \otimes W,s).$$
  \item\label{item:3} Assume that $V$ has a finite direct sum decomposition  $V\simeq \oplus_{i=1}^k V_i \otimes W_i$ into tensor products of compatible systems of $\Gamma_F$-representations. Assume that $V_i$ and $W_i$ are jointly $S$-strongly potentially automorphic for each $i$. Then $L(V,s)$ has meromorphic continuation to all of $\mathbb{C}$, and for any $\ell \in S$, $$\dim V_\ell^{\Gamma_F}=-\ord_{s=1}L(V,s).$$
  \end{altenumerate}
\end{lemma}

\begin{remark}
Lemma \ref{lem:pole} should be known to the experts and the proof idea, using Brauer's induction theorem and known properties of automorphic $L$-functions, is an old one (see e.g. \cite{Tay02,HST10,Har09}).   Notice that Item (\ref{item:1}) also follows as a special case of Item (\ref{item:3}). We keep Item (\ref{item:1}) to illustrate the ideas.
\end{remark}

\begin{proof}
  \begin{altenumerate}
  \item Let $L/F$ be a finite Galois extension such that $V$ is $S$-strongly automorphic over $L$. By Brauer's induction theorem, we may find a virtual decomposition $$\mathbf{1}_{\Gamma_F}=\sum_{j=1}^k c_j \Ind_{\Gamma_{L_j}}^{\Gamma_F} \psi_i,$$ where $c_j\in \mathbb{Z}$, $F\subseteq L_j\subseteq L$ with $L/L_j$ solvable, and $\psi_j$ is a 1-dimensional representation of $\Gal(L/L_j)$ ($j=1,\ldots,k$). Since $V$ is $S$-strongly automorphic over $L$, we know that for each $j$ there exists an isobaric direct sum of cuspidal automorphic representations $\pi_{L_j}=\boxplus_{i=1}^{m_i}\pi_{L_j,i}$ of $\GL_{n}(\mathbb{A}_{L_j})$ and a direct sum decomposition $V|_{\Gamma_{L_j}}=\oplus_{i=1}^{m_j}V_{L_j,i}$ into $\Gamma_{L_j}$-representations such that $$L(V|_{\Gamma_{L_j}},s)=L(s,\pi_{L_j}) \quad L(V_{L_j,i},s)=L(s,\pi_{L_j,i}),$$ and each $\ell$-adic representation $V_{L_j,i, \ell}$ is irreducible for any $\ell \in S$.  Here $L(s,\pi_{L_j})$ is the (incomplete) standard $L$-function as in \cite{GJ72} and has meromorphic continuation to all of $\mathbb{C}$. Hence $$L(V \otimes\Ind_{\Gamma_{L_j}}^{\Gamma_F}\psi_j,s)=L(V|_{\Gamma_{L_j}} \otimes \psi_j,s)=\prod_{i=1}^{m_j}L(V_{L_j,i} \otimes \psi_j,s)=\prod_{i=1}^{m_j}L(s,\pi_{L_j,i} \otimes \chi_j),$$ where $\chi_j$ is the automorphic character on $\GL_1(\mathbb{A}_{L_j})$ associated to $\psi_j$. It follows that $$L(V,s)=L(V \otimes \mathbf{1}_{\Gamma_F},s)=\prod_{j=1}^k\prod_{i=1}^{m_j}L(s, \pi_{L_j,i} \otimes \chi_j)^{c_j}$$ and thus $L(V,s)$ has meromorphic continuation to all of $\mathbb{C}$.

    Since $\pi_{L_j,i} \otimes \chi_j$ is cuspidal, by \cite{JS77} we know that $L(s, \pi_{L_j,i} \otimes \chi_j)$ has no zero or pole at $s=1$, unless $\pi_{L_j,i} \otimes \chi_j$ is the trivial representation in which case it has a simple pole at $s=1$. Hence $-\ord_{s=1}L(V,s)$ equals the number of trivial representations among $\pi_{L_j,i} \otimes \chi_j$ weighted by $c_j$, and so we obtain $$-\ord_{s=1}L(V,s)=\sum_{j=1}^k\sum_{i=1}^{m_j} c_j\dim\Hom_{\Gamma_{L_j}}(\mathbf{1}_{\Gamma_{L_j}}, V_{L_j,i,\ell} \otimes \psi_{j,\ell}),$$ for any $\ell\in S$ by the irreducibility of $V_{L_j,i,\ell}$. This evaluates to $$\sum_{j=1}^k c_j\dim\Hom_{\Gamma_{L_j}}(\mathbf{1}_{\Gamma_{L_j}}, V_\ell|_{\Gamma_{L_j}} \otimes \psi_{j,\ell}),$$ which by the Frobenius reciprocity equals $$\dim \Hom_{\Gamma_F}(\mathbf{1}_{\Gamma_F}, V_\ell)=\dim V_\ell^{\Gamma_F}.$$
  \item Let $L/F$ be a finite Galois extension  such that both $V$ and $W$ are $S$-strongly automorphic over $L$. By the same notation and argument in the proof of Item (\ref{item:1}), we know that for each $j$ there exists an isobaric direct sum of cuspidal representations $\pi_{L_j}=\boxplus_{i=1}^{m_j}\pi_{L_j,i}$ (resp. $\Pi_{L_j}=\boxplus_{i'=1}^{m'_j}\Pi_{L_j,i'}$), together with a corresponding decomposition into $\Gamma_{L_j}$-representations $V|_{\Gamma_{L_j}}\simeq \oplus_{i=1}^{m_j}V_{L_j,i}$ (resp. $W|_{\Gamma_{L_j}}\simeq \oplus_{i'=1}^{m'_j} W_{L_j,i'}$) such that each $\ell$-adic representation $V_{L_j,i, \ell}$ (resp. $W_{L_j,i',\ell}$) is irreducible for any $\ell \in S$. It follows that $$L(V \otimes W,s)=\prod_{j=1}^kL(V  \otimes W \otimes \mathbf{1}_{\Gamma_F},s)=\prod_{j=1}^k\prod_{i=1}^{m_j}\prod_{i'=1}^{m'_j}L(s, \pi_{L_j,i} \times (\Pi_{L_j,i'} \otimes \chi_j))^{c_j},$$ where $L(s, \pi_{L_j,i} \times (\Pi_{L_j,i'} \otimes \chi_j))$ is the (incomplete) Rankin--Selberg $L$-function as in \cite{JPS83}, and thus $L(V \otimes W,s)$ has meromorphic continuation to all of $\mathbb{C}$.

    Since  $\pi_{L_j,i}$ and $\Pi_{L_j,i} \otimes \chi_j$ are cuspidal,  we know that  $L(s, \pi_{L_j,i} \times (\Pi_{L_j,i} \otimes \chi_j))$ has no zero at $s=1$ by \cite{Sha80} (see also \cite[Lemma 3.1]{Mor85}, \cite[p. 721]{Sar04}). Also by \cite[(4.6) and (4.11)]{JS81} (see also \cite[Appendice]{MW89}, \cite[Theorem 2.4]{CP04}), it has no pole at $s=1$, unless $\pi_{L_j,i}\simeq (\Pi_{L_j,i'} \otimes \chi_j)^\vee$ in which case it has a simple pole at $s=1$. The latter happens if and only if $V_{L_j,i}\simeq (W_{L_j,i'} \otimes \psi_{j})^\vee$. Hence $$-\ord_{s=1}L(V,s)=\sum_{j=1}^k\sum_{i=1}^{m_j}\sum_{i'=1}^{m'_j}c_j \dim \Hom_{\Gamma_{L_j}}(\mathbf{1}_{\Gamma_{L_j}}, V_{L_j,i,\ell} \otimes W_{L_j,i',\ell} \otimes \psi_{j,\ell})$$ for any $\ell\in S$ by the irreducibility of $V_{L_j,i,\ell}$ and $W_{L_j,i',\ell}$. This evaluates to $$\sum_{j=1}^kc_j\dim \Hom_{\Gamma_{L_j}}(\mathbf{1}_{\Gamma_{L_j}}, (V_\ell\otimes W_\ell)|_{\Gamma_{L_j}} \otimes \psi_{j,\ell}),$$  which by the Frobenius reciprocity equals $$\dim \Hom_{\Gamma_F}(\mathbf{1}_{\Gamma_F}, V_\ell \otimes W_\ell)=\dim (V_\ell \otimes W_\ell)^{\Gamma_F}.$$
  \item It follows directly from Item (\ref{item:2}) and the factorization $L(V,s)=\prod_{i=1}^k L(V_i \otimes W_i,s)$.\qedhere
  \end{altenumerate}
\end{proof}

\begin{lemma}\label{lem:automorphy}
  Assume that $F$ is a number field.  Let $E_1,E_2,E_3,E_4$ be elliptic curves over $F$. Let $A$  be an abelian surface over $F$.
  \begin{altenumerate}
  \item\label{item:4} If $F$ is totally real or imaginary CM, then $\{\Sym^{k_1}\HH^1(E_{1,\ov F},\Qlb)\}$  and $\{\Sym^{k_2}\HH^1(E_{2,\ov F},\Qlb)\}$ are jointly strongly potentially automorphic for any $k_1, k_2\ge0$.
  \item\label{item:three} If $F$ is totally real or imaginary CM, then $\{\Sym^{k_1}\HH^1(E_{1,\ov F}, \Qlb)\}$ and $\{\Sym^{k_2}\HH^1(E_{2,\ov F}, \Qlb)\otimes \Sym^{k_3}\HH^1(E_{3,\ov F}, \Qlb)\}$ are jointly strongly potentially automorphic for any $k_1\ge0$, $0\le k_2\le 2$, and $0\le k_3\le1$.
  \item\label{item:four} If $F$ is totally real or imaginary CM, then $\{\Sym^{k_1}\HH^1(E_{1,\ov F}, \Qlb)\otimes \Sym^{k_3}\HH^1(E_{3,\ov F}, \Qlb)\}$ and\break $\{\Sym^{k_2}\HH^1(E_{2,\ov F}, \Qlb)\otimes \Sym^{k_4}\HH^1(E_{4,\ov F}, \Qlb)\}$ are jointly strongly potentially automorphic for any $0\le k_1,k_2\le2$ and $0\le k_3,k_4\le1$.
  \item\label{item:5} If $F$ is totally real, then $\{\HH^{k_1}(A_{\ov F}, \Qlb)\}$ and $\{\HH^{k_2}(A_{\ov F}, \Qlb)\}$ are jointly strongly potentially automorphic for any $0\le k_1,k_2\le 4$.
  \end{altenumerate}
\end{lemma}

\begin{proof}
  \begin{altenumerate}
  \item\label{item:7}  If one of $E_1$ or $E_2$ has CM, say $E_1$ has CM, then $\{\Sym^{k_1}\HH^1(E_{1,\ov F},\Qlb)\}$ is automorphic, as an isobaric direct sum of automorphic characters on $\GL_1(\mathbb{A}_F)$, and possibly automorphic inductions of automorphic characters on $\GL_1(\mathbb{A}_K)$ for a quadratic extension $K/F$. In particular, we know that $\{\Sym^{k_1}\HH^1(E_{1,\ov F},\Qlb)\}|_{\Gamma_L}$ is $S$-strongly automorphic over any finite Galois extension $L/F$ and any nonempty set $S$ of primes.  The result follows if $E_2$ also has CM. If $E_2$ has no CM, then  $\{\HH^1(E_{2,\ov F},\Qlb)\}$ is strongly irreducible in the sense defined before \cite[Lemma~7.1.1]{ACC+18} (i.e., for any finite extension $F'/F$, the representation $\HH^1(E_{2,\ov F},\Qlb)|_{\Gamma_{F'}}$ is irreducible for $\ell$ in a Dirichlet density one set of primes), and we can apply \cite[Corollary 7.1.11]{ACC+18} to $\{\Sym^{k_2}\HH^1(E_{2,\ov F},\Qlb)\}$ together with \cite[Proposition 6.5.13]{ACC+18} to obtain the desired joint $S$-strong potential automorphy for a Dirichlet density one set $S$ of primes.  If neither of $E_1$ and $E_2$ has CM, then the desired result follows from the more general \cite[Theorem 7.1.10]{ACC+18} together with \cite[Proposition 6.5.13]{ACC+18}. (In the case $F=\mathbb{Q}$, we may also directly apply \cite[Theorem A (non-CM case) and Theorem A.1 (CM case)]{NT21}).
  \item\label{item:8} By the same argument in Item (\ref{item:7}), there are a finite Galois extension $L/F$ and a Dirichlet density one set $S$ of primes such that $\{\Sym^{k_i}\HH^1(E_{i,\ov F},\Qlb)\}$ is $S$-strongly automorphic over $L$ for any $1\le i \le 3$. Hence by the functorial products for $\GL(2)\times \GL(2)\rightarrow\GL(4)$ (\cite[Theorem M]{Ram00}) and $\GL(2)\times\GL(3)\rightarrow\GL(6)$ (\cite[Theorem A]{KS02a}), we know that $\{\Sym^{k_2}\HH^1(E_{2,\ov F},\Qlb)\otimes \Sym^{k_3}\HH^1(E_{3,\ov F},\Qlb)\}$ is also $S$-strongly automorphic over $L$ for any $0\le k_2\le2$ and $0\le k_3\le1$. The result then follows.
  \item By the same argument in Item (\ref{item:8}), there are a finite Galois extension $L/F$ and a Dirichlet density one set $S$ of primes such that  $\{\Sym^{k_i}\HH^1(E_{i,\ov F},\Qlb)\otimes \Sym^{k_j}\HH^1(E_{j,\ov F},\Qlb)\}$ is $S$-strongly automorphic over $L$ for any $0\le k_i\le2$ and $0\le k_j\le1$, which gives the result.
  \item The result follows from \cite[Theorem 9.3.1]{BCGP21} and its proof.\qedhere
  \end{altenumerate}
\end{proof}

\begin{remark}
For each item of Lemma \ref{lem:automorphy}, the proof  supplies a  Dirichlet density one set $S$ of primes such that the joint $S$-strong potential automorphy holds. Since compatible systems in Lemma \ref{lem:automorphy} come from elliptic curves and abelian surfaces, one should also be able to prove directly that the irreducible conditions required in Definition~\ref{def:strong}~(\ref{item:irr}) hold for all primes $\ell$, and hence  the joint $S$-strong potential automorphy holds for the set $S$ of all primes. For the purpose of the proof of  Theorem \ref{thm:main2} (Tate~II) below, any nonempty $S$ suffices.
\end{remark}

\section{Proof of Theorem \ref{thm:main2} (Tate II)}  Let $1\le r\le \dim X$. Let $V=\{\HH^{2r}(X_{\ov F}, \Qlb(r))\}$. By Theorem \ref{thm:main1} (Tate I), we know from (\ref{eq:Chh}) that $\rank \Chh^r(X)=\dim V_\ell^{\Gamma_F}$ for any prime $\ell$. Thus it remains to show that $\dim V_\ell^{\Gamma_F}=-\ord_{s=1}L(V,s)$ for some prime $\ell$.

\begin{altenumerate}
\item  By the K\"unneth formula and the decomposition of $\HH^1(E_{i,\ov F}, \Qlb)^{\otimes k_i}$ into symmetric powers of  $\HH^1(E_{i, \ov F}, \Qlb)$ ($i=1,2$), we have an isomorphism of semisimple $\Gamma_F$-representations $$\HH^{2r}(X_{\ov F}, \Qlb(r))\simeq \bigoplus_{0\le k_i\le n_i\atop i=1,2}m_{k_1,k_2}\left(\Sym^{k_1}\HH^1(E_{1,\ov F}, \Qlb) \otimes \Sym^{k_2}\HH^1(E_{2,\ov F}, \Qlb)\right)({\scriptstyle\frac{k_1+k_2}{2}}),$$ where $m_{k_1,k_2}\ge0$ are certain multiplicities (nonzero only if $k_1+k_2\le 2r$ is even).   The result then follows from Lemma \ref{lem:pole} (\ref{item:3}) and Lemma \ref{lem:automorphy} (\ref{item:4}).
\item Similarly, set $n_3=1$ then we have an isomorphism of semisimple $\Gamma_F$-representations $$\HH^{2r}(X_{\ov F},\Qlb(r))\simeq \bigoplus_{0\le k_i\le n_i\atop 1\le i\le 3}m_{k_1,k_2,k_3}\left( \otimes_{1\le i\le 3}\Sym^{k_i}\HH^1(E_{i,\ov F}, \Qlb)\right)({\scriptstyle\frac{k_1+k_2+k_3}{2}}),$$ where $m_{k_1,k_2,k_3}\ge0$ are certain multiplicities (nonzero only if $k_1+k_2+k_3\le 2r$ is even). The result then follows from Lemma \ref{lem:pole} (\ref{item:3}) and Lemma \ref{lem:automorphy} (\ref{item:three}).
\item  Similarly, the result follows from Lemma \ref{lem:pole} (\ref{item:3}) and Lemma \ref{lem:automorphy} (\ref{item:four}).
\item  For $X=A$, the result follows from Lemma \ref{lem:pole} (\ref{item:1}) and Lemma \ref{lem:automorphy} (\ref{item:5}). For $X=A^2$, by the K\"unneth formula, we have  an isomorphism of semisimple $\Gamma_F$-representations $$\HH^{2r}(X_{\ov F},\Qlb(r))\simeq \bigoplus_{k_1+k_2=2r\atop 0\le k_1,k_2\le 4} (\HH^{k_1}(A_{\ov F},\Qlb) \otimes \HH^{k_2}(A_{\ov F},\Qlb))(r).$$ The result then follows from Lemma \ref{lem:pole} (\ref{item:3}) and Lemma \ref{lem:automorphy} (\ref{item:5}).
\end{altenumerate}

\begin{remark}
  When $X$ is an abelian surface of the type $\Res_{K/F}E$, where $F$ is totally real, $K/F$ is a quadratic CM extension and $E$ is an elliptic curve over $K$, Tate II was proved in \cite{Vir15} using a similar argument. We also refer to \cite{Joh17,Tay20} for more detailed analysis for $L$-functions of abelian surfaces.
\end{remark}

\subsection*{Acknowledgments} The authors are grateful to G. Boxer, F. Calegari, T. Gee and the anonymous referee for helpful comments. C.~L.~was partially supported by the NSF grant DMS-2101157.  W.~Z.~was partially supported by the NSF grant DMS-1901642.

\newcommand{\etalchar}[1]{$^{#1}$}

\end{document}